\documentclass[a4papar]{amsart}
\usepackage{amscd,amsmath,amssymb,amsthm,amsfonts,epsfig,graphics}

\newtheorem{theorem}{Theorem}[section]
\newtheorem{lemma}{Lemma}[section]
\newtheorem{proposition}{Proposition}[section]
\theoremstyle{definition}
\newtheorem{definition}{Definition}[section]

\newtheorem{example}{Example}[section]
\theoremstyle{remark}
\newtheorem{remark}{Remark}[section]

\newcommand{\la}{\langle}
\newcommand{\ra}{\rangle}
\newcommand{\Gal}{\mathop{\mathrm{Gal}}\nolimits}
\newcommand{\Aut}{\mathop{\mathrm{Aut}}\nolimits}
\newcommand{\Div}{\mathop{\mathrm{div}}\nolimits}

\title[Galois Weierstrass points]{Galois Weierstrass points whose Weierstrass semigroups are generated by two elements}

\date{\today}
\thanks{This work was supported by JSPS KAKENHI Grant Numbers 15K04830 and 16K05094.}

\author{Jiryo Komeda}
\address{Department of Mathematics, Center for Basic Education and Integrated Learning, Kanagawa Institute of Technology, Atsugi, Kanagawa, 243-0292, Japan}
\email{komeda@gen.kanagawa-it.ac.jp}

\author{Takeshi Takahashi}
\address{Faculty of Engineering, Niigata University, Niigata 950-2181, Japan}
\email{takeshi@ie.niigata-u.ac.jp}


\begin{document}
\maketitle

\begin{abstract}
Let $C$ be a nonsingular projective curve of genus $\geq 2$ over an algebraically closed field of characteristic $0$. For a point $P$ in $C$, the Weierstrass semigroup $H(P)$ is defined as the set of non-negative integers $n$ for which there exists a rational function $f$ on $C$ such that the order of pole of $f$ at $P$ equals $n$ and $f$ is regular away from $P$. A point $P$ in $C$ is  referred to as a Galois Weierstrass point if the morphism corresponding to the complete linear system $|aP|$ is a Galois covering, where $a$ is the smallest positive integer in $H(P)$. In this paper, we investigate the number of Galois Weierstrass points whose Weierstrass semigroups are generated by two positive integers.
\end{abstract}

\section{Introduction and Theorem} \label{intro}

A curve refers to a complete nonsingular algebraic curve over an algebraically closed field $k$ of characteristic $0$. A plane curve refers to a (complete nonsingular) curve in $\mathbb{P}^2$.
 
Yoshihara introduced the notion of a Galois point for a plane curve as follows.
\begin{definition}[\cite{MiuraYoshihara2000,Yoshihara2001}]
Let $C$ be a plane curve of degree $d \geq 3$. For a point $P \in \mathbb{P}^2$, the projection $\pi_P: C \rightarrow \mathbb{P}^1$ from $P$ induces an extension of function fields $\pi_P^*: k(\mathbb{P}^1) \hookrightarrow k(C)$. $P$ is referred to as a {\it Galois point} for $C$ if the extension is Galois. Moreover, when $P \in C$ or when $P \not\in C$, the point is said to be an inner or outer Galois point, respectively.
\end{definition}

There is a result on the number of inner Galois points as follows. 
\begin{theorem}[\cite{MiuraYoshihara2000,Yoshihara2001}]\label{Theorem on Galois points}
Let $C$ be a plane curve of degree $d\geq 4$. If $d \geq 5$, then the number of inner Galois points for $C$ equals $0$ or $1$. 
 If $d=4$, then the number of inner Galois points for $C$ equals $0$, $1$, or $4$. Moreover, the number equals $4$ if and only if $C$ is projectively equivalent to the curve $XZ^3 + X^4 + Y^4 =0$, where $(X:Y:Z)$ be a system of homogeneous coordinates of $\mathbb{P}^2$. 
\end{theorem}

\begin{remark}
There are also a result on the number of outer Galois points, and many other results on Galois points. See \cite{MiuraYoshihara2000,Yoshihara2001,Fukasawa2010,OpenProblem} and so on. 
\end{remark}

On the other hand, inner Galois points for a plane curve are characterized as Galois Weierstrass points as follows. Let $\mathbb{Z}_{\geq 0}$ be the set of all non-negative integers. 

\begin{definition}[\cite{MorrisonPinkham1986}]
Let $C$ be a curve of genus $g \geq 2$. A point $P \in C$ is termed a {\it Galois Weierstrass point} (GW point), if the morphism 
$\Phi_{|aP|}: C \rightarrow \mathbb{P}^1$ corresponding to the complete linear system $|aP|$ is a Galois covering, where $a$ is the smallest positive integer of the Weierstrass semigroup
\begin{equation*}
H(P):=\{ n \in \mathbb{Z}_{\geq 0} \mid \exists f \in k(C) \text{ such that } (f)_\infty =nP \}.
\end{equation*}
\end{definition}

We denote $\langle a,b \rangle$ as the numerical semigroup generated by elements $a, b \in \mathbb{N}$.

\begin{theorem}[Theorem 2.3 in \cite{KomedaTakahashi2017}] \label{Theorem: GP is GWP}
If a point $P$ is an inner Galois point for a plane curve $C$ of degree $d \geq 4$, then  $P$ is a GW point with $H(P) = \langle d-1, d \rangle$. 

Conversely, if $P$ is a GW point of a curve $C$ with $H(P)=\langle d-1, d \rangle$ ($d \geq 4$), then $C$ is isomorphic to a plane curve and $P$ is an inner Galois point.
\end{theorem}

\begin{remark}
In \cite{KomedaTakahashi2017}, we also characterized outer Galois points with the notation {\it weak Galois Weierstrass points}. 
\end{remark}

In this paper, we study the number of GW points whose semigroups are generated by two elements, as a generalization of Theorem~\ref{Theorem on Galois points}. Our main theorem is the following.

\begin{theorem} \label{Main Theorem}
Let $a$ and $b$ be integers such that $a \geq 3$, $b \geq a+2$ and the greatest common divisor $\gcd (a,b)=1$.
For a curve $C$ we set 
$$GW_{\la a,b\ra}(C):=\{P\in C\mid P \mbox{ is a GW point with }H(P)=\la a,b\ra\}.$$
Then we have the following: 
\begin{enumerate}
\item If $b\equiv a-1 \pmod{a}$, then $\# GW_{\la a,b\ra}(C)=0$ or $b+1$.
\item If $b\not\equiv a-1 \pmod{a}$, then $\# GW_{\la a,b\ra}(C)=0$ or $1$.
\end{enumerate}
\end{theorem}

\begin{remark}
	In the case that $b = a+1$, by Theorems~\ref{Theorem on Galois points} and \ref{Theorem: GP is GWP}, we have that  
	$\# GW_{\la a, a+1 \ra}(C)= 0$ or $1$ if $a \geq 4$, and $\# GW_{\la 3,4 \ra}(C)=0,1$ or $4$.
\end{remark}

In Section~\ref{Section Preliminary}, we describe some elementary inequalities and preliminary results on GW points. We use these results in the proof of Theorem~\ref{Main Theorem}. In Section~\ref{Section Proof of Main Theorem}, we prove Theorem~\ref{Main Theorem}. In Section~\ref{Section Example}, we describe some examples of GW points. 

\section{Preliminary} \label{Section Preliminary}

In this section, we describe some elementary inequalities and preliminary results on GW points. 

\subsection{Elementary inequality} 

Let $a, r \in \mathbb{N}$ hold $a \geq 2$ and $\gcd (a,r)=1$. For $s \in \mathbb{R}$, we denote by $[s]$ the greatest integer less than or equal to $s$. 
In this subsection, we show the following. 

\begin{proposition} \label{InequalityOfTotalWeight}
Let $w_1$ an $w_2$ be numbers as follows: 
$$w_1 := \frac{a}{2} \sum_{l=1}^{a-1} \left( l+ \left[ \frac{lr}{a} \right] -1 \right) \left( l+ \left[ \frac{lr}{a} \right] \right) + \sum_{l=1}^{a-1}  \left( lr-\left[ \frac{lr}{a} \right] a \right)  \left( l+\left[ \frac{lr}{a} \right] \right)- \frac{g(g+1)}{2};$$
$$w_2 := \frac{a}{2} \sum_{l=1}^{a-1} \left( l+ \left[ \frac{lr}{a} \right] -1 \right) \left( l+ \left[ \frac{lr}{a} \right] \right) + \sum_{l=1}^{a-1}  \left(a-l \right)  \left( l+\left[ \frac{lr}{a} \right] \right)- \frac{g(g+1)}{2},$$
where $g=(a-1)(a+r-1)/2$. Then, $(3a+1)(w_1+(a+r)w_2) > (g-1)g(g+1). $
\end{proposition}

For the proof of Proposition~\ref{InequalityOfTotalWeight}, we state some Lemmas.  

\begin{lemma} \label{bijection}
Let $\varphi: \{ 0, 1, 2, \dots, a-1 \} \rightarrow \{ 0, 1, 2, \dots, a-1 \}$ be a map given by $\varphi(l) \equiv rl \pmod{a}$, i.e., $\varphi(l)= lr-[ lr/a] a$. Then, the map $\varphi$ is bijective. In particular, $\sum_{l=1}^{a-1} \varphi(l)= \sum_{l=1}^{a-1}l=(a-1)a/2$ and $\sum_{l=1}^{a-1} l \varphi(l) \leq \sum_{l=1}^{a-1} l^2 = (a-1)a(2a-1)/6$.
\end{lemma}
\begin{proof}
Obvious.
\end{proof}

\begin{lemma} \label{three inequalities}
We have the following:
\begin{enumerate}
\item $\displaystyle \sum_{l=1}^{a-1} \left[ \frac{lr}{a} \right] =\frac{(r-1)(a-1)}{2};$

\medskip
\item $\displaystyle \sum_{l=1}^{a-1} \,l \left[ \frac{lr}{a} \right] \geq \frac{(r-1)(a-1)(2a-1)}{6};$

\medskip
\item $\displaystyle \sum_{l=1}^{a-1} \left[ \frac{lr}{a} \right] ^2 \geq \frac{(r-1)^2(a-1)(2a-1)}{6a}.$
\end{enumerate}
\end{lemma}
\begin{proof}
Let $\varphi$ be the map stated in Lemma~\ref{bijection}. 
\begin{enumerate}
\item $\sum_{l=1}^{a-1}[lr/a] = \sum_{l=1}^{a-1} (lr/a - \varphi(l)/a) = 1/a (r \sum_{l=1}^{a-1} l - \sum_{l=1}^{a-1} \varphi(l))=(r-1)(a-1)/2$.

\medskip
\item $\sum_{l=1}^{a-1}l [ lr/a ] = \sum_{l=1}^{a-1} l (lr/a-\varphi(l)/a) = r/a \sum_{l=1}^{a-1} l^2 - 1/a \sum_{l=1}^{a-1} l\varphi(l) \geq r/a \sum_{l=1}^{a-1} l^2 - 1/a \sum_{l=1}^{a-1} l^2 = (r-1)(a-1)(2a-1)/6$.

\medskip
\item $\sum_{l=1}^{a-1} [lr/a] ^2 = \sum_{l=1}^{a-1} (lr/a-\varphi(l)/a)^2 = 1/a^2 (r^2 \sum_{l=1}^{a-1} l^2 -2r \sum_{l=1}^{a-1} l \varphi(l) + \sum_{l=1}^{a-1} \varphi(l)^2 ) \geq 1/a^2 (r^2 \sum_{l=1}^{a-1} l^2 -2r \sum_{l=1}^{a-1} l^2 + \sum_{l=1}^{a-1} l^2) = (r-1)^2(a-1)(2a-1)/(6a)$.
\end{enumerate}
\end{proof}

Remark that, in each formula in Lemma~\ref{three inequalities}, we can obtained the right-hand side from the left-hand side by replacing $[lr/a]$ to $(r-1)l/a$. 

\begin{lemma} \label{w_1 >= w_2}
$w_1 \geq w_2$. 
\end{lemma}
\begin{proof}
Remark that $lr-[lr/a]a = \varphi (l)$ where $\varphi$ is the bijective map stated in Lemma~\ref{bijection}, and the map $l \mapsto l+[lr/a]$ is strictly increasing. Hense, $$\sum_{l=1}^{a-1}  \left( lr-\left[ \frac{lr}{a} \right] a \right)  \left( l+\left[ \frac{lr}{a} \right] \right) \geq \sum_{l=1}^{a-1}  \left(a-l \right)  \left( l+\left[ \frac{lr}{a} \right] \right). $$
\end{proof}

\begin{remark} \label{Remark w_1 > w_2}
    We see that $r \equiv a-1 \pmod{a}$ if and only if $w_1 = w_2$. 
\end{remark}

\begin{proof}[Proof of Proposition~\ref{InequalityOfTotalWeight}]
By Lemmas~\ref{three inequalities} and \ref{w_1 >= w_2}, we have the following.
\begin{align*}
 & \hspace{3ex} (3a+1)(w_1+(a+r)w_2) \\
 & \geq (3a+1)(a+r+1) w_2 \\ 
 & = \frac{1}{2}(3a+1)(a+r+1) \left( \sum_{l=1}^{a-1} (l+[lr/a])(a+(a-2)l+a[lr/a]) -g(g+1) \right) \\
 & \geq \frac{1}{2} (3a+1)(a+r+1) \left( \sum_{l=1}^{a-1} (l+(r-1)l/a)(a+(a-2)l+(r-1)l) -g(g+1) \right) \\
 & =\frac{g}{4} (a^2+ar-2a+r-3) (a^2+ar+4a/3+r/3+1/3).
\end{align*}

On the other hand, we have the following.
$$ \hspace{3ex} (g-1)g(g+1) = \frac{g}{4}(a^2+ar-2a-r-1)(a^2+ar-2a-r+3).$$
Hence, we conclude. 
\end{proof}

\subsection{GW point}

First, we state the structure of the function field of a curve having a GW point with a Weierstrass semigroup generated by two integers. 

\begin{lemma}[Lemma~2.3 in \cite{KomedaTakahashi2017}] \label{KT2017}
Let $P$ be a GW point of a curve $C$ with $H(P)= \la a,b \ra$. Then, for any $x \in k(C)$ with $(x)_{\infty} = aP$, there exists $y \in k(C)$ with $(y)_\infty = bP$ such that $k(C)=k(x,y)$ and $y^a= \prod_{i=1}^b (x-c_i)$, where each $c_i$ is a distinct element of $k$.  
\end{lemma}

\begin{remark} \label{genus of <a,b>}
	Since
	$\mathbb{Z}_{\geq 0} \setminus \la a,b \ra = \{ jb - na \mid j \in \{1, \dots, a-1 \}, n \in \mathbb{N}, jb-na > 0 \},  $
	we have that $\#( \mathbb{Z}_{\geq 0} \setminus \la a,b \ra )= \sum_{j=1}^{a-1} [jb/a]= (a-1)(b-1)/2$ by Lemma~\ref{three inequalities}. For a point $P$ in a curve $C$, we have that 
	$\# (\mathbb{Z}_{\geq 0} \setminus H(P)) = g(C)$, where $g(C)$ is the genus of $C$. Hence, 
	if a curve $C$ has a point $P$ with $H(P)= \la a,b \ra$, then $g(C)=(a-1)(b-1)/2$. 
\end{remark}

Let $\pi:C\rightarrow C'$ be a cyclic covering of curves with a total ramification point $P$ over $P'$. We can find $H(P')$ from some $H(P)$ as follows. 

\begin{lemma} \label{total}
Let $\pi:C\rightarrow C'$ be a cyclic covering of degree $t$ with a total ramification point $P$ over $P'$. 
Then, we have
$H(P')=\{h' \in \mathbb{Z}_{\geq 0} \mid th'\in H(P)\}.$
\end{lemma}
\begin{proof}
Let $\sigma$ be an automorphism of $C$ such that $C'=C/\la \sigma \ra$.

Let $h'\in H(P')$, i.e., there exists a rational function $f$ on $C'$ such that $(f)_{\infty}=h'P'$.
Then, we get $(\pi^*f)_{\infty}=th'P$, where we set $\pi^*f=f\circ \pi$.
Hence, we obtain
$H(P')\subset \{h'\mid th'\in H(P)\}.$

Conversely, let $h'\in \mathbb{Z}_{\geq 0}$ with $th'\in H(P)$.
Then we have a rational function ${\tilde f}$ on $C$ with $({\tilde f})_{\infty}=th'P$.
Consider
$F={\tilde f}+\sigma^*{\tilde f}+(\sigma^*)^2{\tilde f}+\dots+(\sigma^*)^{t-1}{\tilde f},$
which is a rational function on $C$.
Then we get
$\sigma^*F=\sigma^*{\tilde f}+(\sigma^*)^2{\tilde f}+(\sigma^*)^3{\tilde f}+\dots+(\sigma^*)^{t-1}{\tilde f}+{\tilde f}=F.$ 
Hence, we get $F\in k(C')$.
Let
$${\tilde f}=\frac{1}{z^{th'}}+\frac{c_{-th'+1}}{z^{th'-1}}+\dots,$$
where $z$ is a local parameter at $P$ with $\sigma^*z=\zeta z$, here $\zeta$ is a primitive $t$-th root of unity.
Then we obtain
$\sigma^*{\tilde f}= 1/z^{th'}+\dots,$
which implies that
$F=t/z^{th'}+\dots.$
Then we get $(F)_{\infty}=th'P$ on $C$,
which implies that
$(F)_{\infty}=h'P'$ on $C'$.
Hence, we obtain $h'\in H(P')$.
\end{proof}

\begin{lemma} \label{gen2}
Let $a$, $b$ and $t$ are positive integers with $a<b$, $\gcd (a,b)=1$ and $t|a$.
Then, we have $\{h' \in \mathbb{Z}_{\geq 0} \mid th'\in \la a,b\ra\}=\la a/t, b \ra.$
\end{lemma}
\begin{proof}
Let $\displaystyle h'\in \la a/t,b\ra$, i.e., $\displaystyle h'=n\cdot a/t+mb$ with $n\in \mathbb{Z}_{\geq 0}$ and $m\in \mathbb{Z}_{\geq 0}$.
Hence, we get $th'=na+mtb\in \la a,b\ra$.
Thus, we obtain $\displaystyle \la a/t,b\ra\subset \{h'\mid th'\in \la a,b\ra\}$.

Let $h'$ with $th'\in \la a,b\ra$, i.e., $th'=na+mb$ with $n,m\in \mathbb{Z}_{\geq 0}$.
In view of $t|a$ we have $a=td$ for some $d$.
Hence, we get
$mb=th'-na=th'-ntd=t(h'-nd).$
Since $\gcd (a,b)=1$ and $t|a$, we have $\gcd (t,b)=1$.
Thus, we get $t|m$, i.e., $m=td'$ for some $d'$.
Hence, we have
$th'=na+td'b$, which implies that
$\displaystyle h'=n\cdot a/t+d'b\in \la a/t, b \ra.$
Thus, we obtain
$\{h'\mid th'\in \la a,b\ra\}\subset \la a/t,b\ra.$
\end{proof}

For a cyclic covering which is given by a GW point whose semigroup is generated by two integers, we can compute the Weierstrass semigroup of a ramification point as follows. 

\begin{lemma} \label{Ramification Points}
Let $P$ be a GW point with $H(P)=\la a, b \ra$ on a curve $C$. Then, the ramification points of the map $\Phi_{|aP|}$ are $P$ and other $b$ points (say $Q_1, \dots, Q_b$) which are totally ramified. Furthermore, the Weierstrass semigroup $H(Q_i)$ of every $Q_i$ is 
$\la a,b \ra $ if $b+1 \equiv 0 \pmod{a}$, or
$$\la \{ a \} \cup \{ l(b+a-\bar{r}) - [l(a-\bar{r}+1)/a]a \mid l=1, \dots, a-1 \} \ra $$
if $b+1 \equiv \bar{r} \pmod{a}$ and $0<\bar{r}<a$.
\end{lemma}
\begin{proof}
By Lemma~\ref{KT2017}, we have $x, y \in k(C)$ such that $(x)_{\infty}=aP$, $(y)_{\infty}=bP$, $k(C)=k(x,y)$ and $y^a=\prod_{i=1}^{b}(x-c_i)$, where each $c_i$ is a distinct element of $k$. Then, we may assume that the map $\Phi_{|aP|}$ is given by the function $x$ and the ramification point $Q_i$ holds $x(Q_i)=c_i$ and $y(Q_i)=0$. Let $m$ and $\bar{r}$ be the integers such that $b+1=ma+\bar{r}$ and $0 < \bar{r} \leq a$. Let $\tilde{x}_i: = 1/(x-c_i)$ and $\tilde{y}_i:=\tilde{x}_i^{m+1} y $. Since $\Div(x-c_i) = aQ_i-aP$ and  $\Div(y) = Q_1 + \dots + Q_b - bP$, we have that $\Div(\tilde{x}_i) = aP - a Q_i$ and $\Div(\tilde{y}) = Q_1+ \dots + Q_b -bP - (m+1) (aQ_i-aP) = (Q_1 + \dots + Q_b -Q_i) + (a-\bar{r}+1)P - (a(m+1)-1)Q_i$. Moreover, for $l = 1, \dots, a-1$ and  $N_l:=[l(a-\bar{r}+1)/a]$, we have that $(\tilde{y}^l/\tilde{x}^{N_l})_{\infty} = (l(b+a-\bar{r})-N_la)Q_i$. Let  
$$S:= \{ a \} \cup \{ l(b+a-\bar{r}) - [l(a-\bar{r}+1)/a]a \mid l=1, \dots, a-1 \}, $$
and $H:=\la S \ra$. Then, now we have $H \subset H(Q_i)$. 

Let  
\begin{multline*}
  G:= \{ m_{l,n} \mid m_{l,n} = l(b+a-\bar{r}) - [l(a-\bar{r}+1)/a]a -na >0, \\ l \in \{1, \dots, a-1\},  n \in \mathbb{N}\}.  
\end{multline*}
Then, $G \supset \mathbb{Z}_{\geq 0} \setminus H \supset \mathbb{Z}_{\geq 0} \setminus H(Q_i)$. By Lemma~\ref{three inequalities}, the number of elements of $G$ is equal to 
$$\sum_{l=1}^{a-1} \left[ l(b+a-\bar{r})/a \right] - \sum_{l=1}^{a-1} \left[ l(a-\bar{r}+1)/a \right] = (a-1)(b-1)/2.$$
The number of elements of $\mathbb{Z}_{\geq 0} \setminus H(Q_i)$ is equal to the genus of $C$, i.e., $g(C)=(a-1)(b-1)/2$ by Remark~\ref{genus of <a,b>}. Hence, $G = \mathbb{Z}_{\geq 0} \setminus H = \mathbb{Z}_{\geq 0} \setminus H(Q_i)$, so $H(Q_i)=H$.

If $\bar{r}=a$, then, we have $H=\la a, b\ra$.
\end{proof}

\begin{remark} \label{Remark standard basis}
For a numerical $a$-semigroup $H$, i.e., $H$ is a submonoid of $(\mathbb{Z}_{\geq 0}, +)$ such that $\#(\mathbb{Z}_{\geq 0} \setminus H) < \infty$ and $a$ is the smallest positive integer in $H$, the subset $\{ a, s_1, \dots, s_{a-1}\}$, where $s_i = \min \{t \in H \mid t \equiv i \pmod{a} \}$, is referred to as the standard basis for $H$ (Section~2 in \cite{RimVitulli1977}).  
In the proof of Lemma~\ref{Ramification Points},  $S$ is the standard basis for $H(Q_i)$. The standard basis for $H(P)=\la a,b \ra$ is $\{a, b, 2b, \dots, (a-1)b \}$. We see that $b \equiv a-1 \pmod{a}$ (i.e., $\bar{r}=a$) if and only if $H(Q_i) = \la a, b \ra$.
\end{remark}

\section{Proof of Theorem~\ref{Main Theorem}} \label{Section Proof of Main Theorem}

In this section we prove Theorem~\ref{Main Theorem}. 

\subsection{The case that $a+2 \leq b \leq 2a-2$} 

We prove Theorem~\ref{Main Theorem} for the case that $a+2 \leq b \leq 2a-2$. In this subsection, we assume that $a \geq 5$, $2 \leq r \leq a-2$, $\gcd (a,r)=1$ and $b=a+r$.

\begin{lemma} \label{weight1}
Let $H:= \la a, b \ra$. Then, its Weierstrass weight is equal to the value $w_1$ which is stated in Proposition~\ref{InequalityOfTotalWeight}.  
\end{lemma}
\begin{proof}
The standard basis for $H$ is $\{ a, a+r, 2(a+r), \dots, (a-1)(a+r) \}$. Hence, the gap sequence of $H$ is the following.
$$\{ \, l(a+r)-na \mid l \in \{1,\dots, a-1 \}, \, n \in \mathbb{N}, \, l(a+r)-na >0 \, \}.$$
Hence, the weight of the gap sequence is the following.
\begin{eqnarray*}
    &&\sum_{l=1}^{a-1} \sum_{n=1}^{\left[ \frac{l(a+r)}{a}\right]} (l(a+r)-na) -\frac{g(g+1)}{2} \\
    &=&\sum_{l=1}^{a-1} \sum_{n=1}^{\left[ \frac{l(a+r)}{a}\right]} \left( lr-\left[\frac{lr}{a}\right]a + \left( \left[ \frac{l(a+r)}{a}\right]  -n\right) a\right)  -\frac{g(g+1)}{2} \\
    &=&\sum_{l=1}^{a-1} \sum_{n=1}^{\left[ \frac{l(a+r)}{a}\right]} \left( lr-\left[\frac{lr}{a}\right]a \right) + \sum_{l=1}^{a-1} \sum_{n=0}^{\left[ \frac{l(a+r)}{a}\right]-1} \left( n a\right)  -\frac{g(g+1)}{2} \, = \, w_1.
\end{eqnarray*}
\end{proof}

Remark that $b+1= a+r+1 \equiv r+1 \pmod{a}$. Let $\bar{r}:=r+1$. 

\begin{lemma} \label{weight2}
Let $H:=\la \{ a \} \cup \{ l(b+a-\bar{r}) - [l(a-\bar{r}+1)/a]a \mid l=1, \dots, a-1 \} \ra $. 
Then, its Weierstrass weight is equal to the value $w_2$ which is stated in Proposition~\ref{InequalityOfTotalWeight}. 
\end{lemma}
\begin{proof}
We have that $l(b+a-\bar{r}) - [l(a-\bar{r}+1)/a]a = l(2a-1) - [l(a-r)/a]a = (a-l) + la+[lr/a]a$. 
By Remark~\ref{Remark standard basis}, the gap sequence of $H$ is
$$ \{ \, m_{l,n} \mid  m_{l,n}=(a-l) + la+[lr/a]a - na >0 , \, l \in \{ 1,2, \dots, a-1 \}, \, n \in \mathbb{N} \, \}.$$
By the calculation similar to that in the proof of Lemma~\ref{weight1}, we have that the weight of the gap sequence equals $w_2$. 
\end{proof}

Let $\delta:=\# GW_{\la a, b \ra}(C)$.

\begin{lemma} \label{rough bound} 
$\delta = 0,1,a+1$ or $2a+1$. 
\end{lemma}
\begin{proof}
Assume that there exist two GW points $P$ and $Q$ with $H(P)=H(Q)=\la a, b \ra$. By Lemma~\ref{Ramification Points} and Remark~\ref{Remark standard basis}, the map $\Phi_{|aP|}$ is a cyclic covering of degree $a$ and is not ramified at $Q$. Let $\sigma \in \Aut(C)$ be a generator of the cyclic group $\Gal (\Phi_{|aP|}) := \{ \tau \in \Aut (C) \mid \Phi_{|aP|} \circ \tau = \Phi_{|aP|} \}$. Then, we can find $a+1$ GW points with the Weierstrass semigroup $\la a,b \ra$, that is, these are $P$ and $\sigma^i(Q)$ ($0 \leq i \leq a-1$). 

Assume that there exist $a+2$ GW points $P$, $\sigma^i(Q)$ ($0 \leq i \leq a-1$) and $R$ with the Weierstrass semigroup $\la a,b \ra$. By the same argument as above, we can find $2a+1$ GW points with the Weierstrass semigroup $\la a,b \ra$. 

Assume that $\delta \geq 2a+2$. By the same argument as above, we see that $\delta \geq 3a+1$. By Lemmas~\ref{weight1} and \ref{Ramification Points}, the cyclic covering with respect to each GW point has $b+1$ total ramification points, and the sum of weights of these $b+1$ points is equal to $w_1+ b w_2$. Since the total weights of Weierstrass points is equal to $(g-1)g(g-1)$ (Proposition III.5.10 in \cite{FarkasKra1992}), we have that $(3a+1)(w_1+b w_2) \leq (g-1)g(g-1)$. This contradicts to Proposition~\ref{InequalityOfTotalWeight}. 
\end{proof}

Here we prove Theorem~\ref{Main Theorem} for the case that $a+2 \leq b \leq 2a-2$, that is, we show the following Lemma.

\begin{lemma}
$\delta = 0$ or $1$. 
\end{lemma}
\begin{proof}
If $\delta > 1$, then $\delta = a+1$ or $2a+1$ by Lemma~\ref{rough bound}.

Assume that $\delta = a+1$. Let $P_i$ ($0 \leq i \leq a$) be these $a+1$ GW points. Then, by the argument in the proof of Lemma~\ref{rough bound}, we have $aP_0 \sim \sum_{i=1}^{a} P_i$, where $\sim$ means linear equivalence. We also have $aP_1 \sim P_0 + \sum_{i=2}^{a} P_i$. Hence, $(a+1)P_0 \sim \sum_{i=0}^{a} P_i \sim (a+1) P_1$. However, $a+1$ is a gap of $H(P_0)=\la a,a+r \ra$. This is a contradiction. 

Assume that $\delta = 2a+1$. Let $P_i$ ($0 \leq i \leq 2a$) be these $2a+1$ GW points. Then, by the argument in the proof of Lemma~\ref{rough bound}, we may assume $aP_0 \sim \sum_{i=1}^{a} P_i \sim \sum_{i=a+1}^{2a} P_i$. Hence, $2aP_0 \sim \sum_{i=1}^{2a} P_i$. We also have $2aP_1 \sim P_0 + \sum_{i=2}^{2a} P_i$. Hence, $(2a+1) P_0 \sim \sum_{i=0}^{2a} P_i \sim (2a+1) P_1$. However, $2a+1$ is a gap of $H(P_0)=\la a,a+r \ra$. This is a contradiction.
\end{proof}

\subsection{Proof of Theorem~\ref{Main Theorem} for the case that $2a-1 \leq b$} \label{Section Proof 2a-2 < b} 

We prove Theorem~\ref{Main Theorem} for the case that $2a-1 \leq b$. 

\begin{lemma} \label{uniq}
Let $P$ be a GW point on a curve $C$ with $H(P)=\la a,b\ra$, where $3\leqq a$, $2a<b$ and $\gcd (a,b)=1$. Then, $|aP|$ is only one base-point free linear system on $C$ with projective dimension $1$ and degree $a$.
\end{lemma}
\begin{proof}
We set $g_a^1=|aP|$. Assume that there is another base-point free linear system $h_a^1$ on $C$ with projective dimension $1$ and degree $a$ such that $g_a^1 \not\sim h_a^1$. 
Since $2a<b$, by Remark~\ref{genus of <a,b>}, we have
$$g(C)=\frac{(a-1)(b-1)}{2}\geqq \frac{(a-1)(2a+1-1)}{2}=a(a-1)>(a-1)^2.$$
Hence, by Castelnuovo's inequality, there exists a morphism $\varphi:C\rightarrow C'$ of degree $t$ with $1<t$ and $t|a$ such that　
$g_a^1=g_{a/t}^1 \circ \varphi \mbox{ and } h_a^1=h_{a/t}^1\circ \varphi,$
where $g_{a/t}^1$ and $h_{a/t}^1$ are base-point free linear systems with projective dimension $1$ and degree $a/t$, and for a base-point free linear system $g_s^1$ we also denote by $g_s^1$ the morphism $C\rightarrow \mathbb{P}^1$ defined by $g_s^1$.
Since $P$ is a Galois Weierstrass point, the morphism $g_a^1$ is a cyclic covering. 
Hence, the morphisms $\varphi$ and $\displaystyle g_{a/t}^1$ are also cyclic coverings of degree $t$ and $a/t$ respectively.
We set $P'=\varphi(P)$.
Then $P$ is a total ramification point over $P'$.
By Lemmas~\ref{total} and~\ref{gen2} we obtain
$\displaystyle H(P')=\left\la a/t, b \right\ra$.
We have
$$g(C')=\frac{(a/t-1)(b-1)}{2}\geqq \frac{(a/t-1)(2a+1-1)}{2}=a\left(a/t-1\right)>\left(a/t-1\right)^2.$$
Since $\displaystyle g_{a/t}^1$ is not linearly equivalent to $\displaystyle h_{a/t}^1$,
by Castelnuovo's inequality there exists a morphism $\varphi_1 : C'\rightarrow C''$ of degree $t_1$ with $1 < t_1$ and $\displaystyle t_1 \Big|(a/t)$ such that 
$g_{a/t}^1=g_{a/(t_1t)}^1\circ \varphi_1 \mbox{ and } h_{a/t}^1=h_{a/(t_1t)}^1 \circ \varphi_1$. 
Since $\displaystyle g_{a/t}^1$ is a cyclic coverings, so are the morphisms $\varphi_1$ and $g_{a/(t_1t)}^1$.
We set $P''=\varphi_1(P')$.
Then we get
$\displaystyle H(P'')= \la {a/(t_1t)},b \ra$.
We can continue this process infinitely.
Since $a$ is bounded, this is a contradiction.
\end{proof}

By Remark~\ref{Remark standard basis} and Lemmas~\ref{uniq}, we conclude Theorem~\ref{Main Theorem} for the case that $2a < b$. 

\begin{lemma} \label{b=2a-1}
Let $C$ be a curve and $P$ be a GW point on $C$ with $H(P)=\la a,2a-1\ra$, where $3 \leq a$. If $Q$ is a point on $C$ such that $H(Q)$ is an $a$-semigroup, then the divisor $aQ$ is linearly equivalent to $aP$.
\end{lemma}
\begin{proof}
Assume that $aQ \not\sim aP$, where $\sim$ means linear equivalence.
Let $x,y\in k(C)$ such that $(x)_{\infty}=aP$ and $(y)_{\infty}=aQ$.
By the assumption, the field $k(x)$ is properly contained in $k(x,y)$.

Let $\varphi: C \rightarrow \mathbb{P}^1\times \mathbb{P}^1$ be the morphism determined by the morphisms $\Phi_{|aP|}:C\rightarrow \mathbb{P}^1$ and $\Phi_{|aQ|}:C\rightarrow \mathbb{P}^1$. Let $C'$ be the nonsingular curve that is birational to the image $\varphi(C)$. We consider the case where $k(x,y)=k(C')$ is properly contained in $k(C)$. By the same way to the proof of Proposition~\ref{uniq}, we get a contradiction. Thus we obtain $k(C)=k(x,y)$. Namely, $C$ is birational to $\varphi(C)$. Since the arithmetic genus of $\varphi(C)$ equals $g(C)=(a-1)^2$, $\varphi(C)$ is nonsingular. Thus, $C$ is embedded in $\mathbb{P}^1\times \mathbb{P}^1$. 

We set $L:=\{\mbox{a point}\}\times \mathbb{P}^1\in {\rm Pic}(\mathbb{P}^1\times \mathbb{P}^1)$ and $M:=\mathbb{P}^1\times \{\mbox{a point}\} \in {\rm Pic}(\mathbb{P}^1\times \mathbb{P}^1)$. 
Then we have $C\sim aL+aM$. Moreover, we obtain $L|_C\sim aP$ and $M|_C\sim aQ$. 
We have an exact sequence 
$0\rightarrow {\mathcal O}_{\mathbb{P}}(-C)\rightarrow {\mathcal O}_{\mathbb{P}}\rightarrow {\mathcal O}_C\rightarrow 0$, where we set $\mathbb{P}=\mathbb{P}^1\times \mathbb{P}^1$.
Hence, we get an exact sequence
$$0\rightarrow {\mathcal O}_{\mathbb{P}}(nL-C)\rightarrow {\mathcal O}_{\mathbb{P}}(nL)\rightarrow {\mathcal O}_C(naP)\rightarrow 0.$$
Since $nL-C\sim nL-aL-aM=-aM+(n-a)L$, by K\"{u}nneth formula, we obtain
$$H^i(nl-C)=H^i(-aM+(n-a)L)=\oplus_{s+t=i}H^s(-a)\otimes H^t(n-a), $$
where for an integer $s$ we set $H^i(s):=H^i({\mathcal O}_{\mathbb{P}^1}(s))$.
Hence, we get
$$H^0(nL-C)=H^0(-a)\otimes H^0(n-a)=0$$
and
$$H^1(nL-C)=(H^1(-a)\otimes H^0(n-a))\oplus (H^0(-a)\otimes H^1(n-a))=H^1(-a)\otimes H^0(n-a),$$
which implies that 
$h^1(nL-C)=h^0(a-2)h^0(n-a)$.
Thus, we have
$h^1(nL-C)=0$ if $n\leqq a-1$, and $h^1(nL-C)=(a-1)(n-a+1)$ if $n\geqq a$.
Therefore, we have an exact sequence 
$0\rightarrow H^0(nL)\rightarrow H^0(C,naP)\rightarrow 0$ 
for $n\leqq a-1$.
Hence, we get
$h^0(naP)=h^0(nL)=n+1$ for $n\leqq a-1$.
Since $g(C) = (a-1)^2$ equals $\#(\mathbb{Z}_{\geq 0} \setminus H(P))$, we obtain
$$H(P)=\{0,a,2a,\ldots,(a-1)a\}\cup\{h\in \mathbb{Z}_{\geq 0}\mid h\geqq (a-1)a+1\},$$
which is distinct from the numerical semigroup $\la a,2a-1\ra$.
This is a contradiction.
\end{proof}

By Remark~\ref{Remark standard basis} and Lemma~\ref{b=2a-1}, we conclude Theorem~\ref{Main Theorem} for the case that $b=2a-1$. Now we complete the proof of Theorem~\ref{Main Theorem}.

\section{Example} \label{Section Example}

We state some examples of Theorem~\ref{Main Theorem}. 

\begin{example} \label{Example <3,5>}
Let $C' \subset \mathbb{P}^2$ be a singular plane curve defined by the equation $Y^3Z^2 = \prod_{j=1}^5 (X - j Z)$, where $(X:Y:Z)$ is a system of homogeneous coordinates of $\mathbb{P}^2$. Then, $C'$ has only one singular point $P':=(0:1:0)$, which is a cusp. Remark that $P'$ is an extendable Galois point for $C'$ (For the definition of an extendable Galois point, see \cite{FukasawaMiura2014}), in particular, $k(C')/\pi_{P'}^*(k(\mathbb{P}^1))$ is Galois. Let $\rho: C \rightarrow C'$ be the resolution of singularities, i.e., the composition of blowups over $P'$, and $P:=\rho^{-1}(P')$. The genus of $C$ equals $g(C)=4$.

We see that $H(P)=\la 3,5\ra$ as follows. Let $x$, $y$ $\in k(C)$ be rational functions induced by $X/Z$ and $Y/Z$, respectively. Then, $(x)_{\infty} = 3P$ and $(y)_{\infty}=5P$. Hence, $H(P) \supset \la 3,5 \ra$. The number of elements of the gap sequence at $P$, i.e., $\# (\mathbb{Z}_{\geq 0} \setminus H(P))$, equals $g(C)=4$ and the number $\# (\mathbb{Z}_{\geq 0} \setminus \la 3,5 \ra)$ also equals $4$. So we have $H(P)=\la 3,5\ra$.

The morphism $\Phi_{|3P|}: C \rightarrow \mathbb{P}^1$ corresponding to the linear system $|3P|$ is the composition $\pi_{P'} \circ \rho$. Hence, $\Phi_{|3P|}$ is a Galois covering and $P$ is a GW point. By Lemma~\ref{Ramification Points}, we have that the Weierstrass semigroup $H(Q_j)$ of a ramification point $Q_j:=\rho^{-1}((j:0:1))$ of $\Phi_{|3P|}$, $j=1, \dots, 5$, is  also $\la 3,5 \ra$. Every $Q_j$ is a GW point.	 By Theorem~\ref{Main Theorem}, we have $GW_{\la 3,5 \ra}(C) = \{ P, Q_1, \dots, Q_5\}$
\end{example}

\begin{example}
Let $C' \subset \mathbb{P}^2$ be a singular plane curve defined by the equation $Y^3Z^2 = \prod_{j=1}^7 (X - j Z)$, where $(X:Y:Z)$ is a system of homogeneous coordinates of $\mathbb{P}^2$. Let $\rho: C \rightarrow C'$ be the resolution of singularities. By an similar argument to that in  Example~\ref{Example <3,5>}, we have that $P:=\rho^{-1}((0:1:0))$ is a GW point with $H(P) = \la 3,7 \ra$. Every ramification point $Q_j:=\rho^{-1}((j:0:1))$ of $\Phi_{|3P|}$, $j=1, \dots, 7$, is also a GW point. However, $H(Q_i) = \la 3,8,13 \ra$ by Lemma~\ref{Ramification Points}. By Theorem~\ref{Main Theorem}, $GW_{\la 3,7 \ra}(C)=\{P\}$. 
\end{example}


\end{document}